\documentclass{svjour3}                     
\smartqed  
\usepackage{amsmath,amsfonts,amssymb,amscd}

\usepackage{graphicx}
\begin{document}

\title{Optimal regularity for the Signorini problem
}

\author{Nestor Guillen}


\institute{Nestor Guillen \at
              Department of Mathematics. University of Texas at Austin\\
              \email{nguillen@math.utexas.edu}           
           }

\date{Received: date / Accepted: date}

\maketitle

\begin{abstract}
We prove under general assumptions that solutions of the thin obstacle or Signorini problem in any space dimension achieve the optimal regularity $C^{1,1/2}$. This improves the known optimal regularity results by allowing the thin obstacle to be defined in an arbitrary $C^{1,\beta}$ hypersurface, $\beta>1/2$, additionally, our proof covers any linear elliptic operator in divergence form with smooth coefficients. The main ingredients of the proof are a version of Almgren's monotonicity formula and the optimal regularity of global solutions.\\


\end{abstract}

\section{Introduction} \label{intro}

In this note we look at solutions of the Signorini or thin obstacle problem and make some remarks about their optimal regularity. The thin obstacle problem consists in minimizing a given functional (associated to an elliptic operator) among all functions that remain above a given obstacle which is defined only on a thin set of the domain, i.e. on a set of codimension 1. More specifically, we are given the following data

\begin{equation}\label{thinobs}
	\begin{array}{rl}
	 \mbox{a functional } & \mathcal{F}(u)=\int_{\Omega} \frac{1}{2}\left ( a(x) \nabla u , \nabla u\right ) +u(x)f(x)dx \\
	\mbox{a hypersurface}  & M \subset \mathbb{R}^n \\
	\mbox{functions } & \phi: M \to \mathbb{R, }\mbox{ }  g: \partial \Omega \to \mathbb{R}\\
	\end{array}\end{equation}

then the problem is to find the minimizer of $\mathcal{F}$ among all $u \in H^1(\Omega)$ wich agree with $g$ on $\partial \Omega$ and such that $u \geq \phi$ on $M \cap \Omega$, both understood in the trace sense. We extend the optimal regularity achieved previously under more restricted conditions (see below), our main result is

\begin{theorem}\label{thm:main} 
Assume that $M$ and $\phi$ are both of class $C^{1,\beta}$ where  $\beta > 1/2$, $a(x)$ is an uniformly elliptic matrix of class $C^{1,\gamma}$, (for any $\gamma>0$) and $f(x)$ is any H\"older continuous function, then the solution of the thin obstacle problem is a $C^{1,1/2}$ function on both sides of the hypersurface $M$. 

\end{theorem}

This extends recent results by Athanasopoulos and Caffarelli \cite{AC04} and Caffarelli, Salsa and Silvestre \cite{CSS08}, the specific improvement consists in dropping the assumption that $M$ is a hyperplane and allowing for a more general divergence operator and not only the Laplacian. We follow closely the approach in \cite{CSS08} which relies on a modified Almgren's monotonicity formula, here we use yet another modification of this formula.\\

The thin obstacle problem is closely related (when $n=2$) to the study of the equilibrium of an elastic membrane that rests above a very thin object, a problem better known as the Signorini problem. It also appears in many other questions of continuum mechanics in both stationary and non-stationary problems, for instance temperature control and flow through semi-permeable walls, as explained in the book of Duvaut and Lions~\cite{DL72}. We recall a few of the classical works on the regularity for this problem: in ~\cite{F77}, Frehse proves the Lipschitz continuity  in all dimensions; for the 2 dimensional case, Richardson~\cite{R78} showed solutions are $C^{1,1/2}$, which is the optimal regularity as it is well known; later, Caffarelli ~\cite{C79} proved solutions were $C^{1,\alpha}$ in all dimensions for some small $\alpha$, the book by Friedman \cite{F82} discusses this regularity result in detail.\\

Other applications appear in stochastic control, which require solving an obstacle problem for a fractional power of the Laplacian, the regularity for this obstacle problem was studied by Silvestre in \cite{S07} using methods from potential theory, there he obtained almost optimal regularity for a general obstacle and optimal regularity for global solutions. A different approach is the one taken in \cite{CSS08} which is based on the extension technique of Caffarelli and Silvestre (see \cite{CS06}), there it is shown that the obstacle problem for $\Delta^s$ in $\mathbb{R}^{n-1}$ ($0<s<1$) is equivalent to the thin obstacle problem for an operator $L_s$ in $\mathbb{R}^n$ with obstacle defined on an hyperplane $M$ (which is identified with $\mathbb{R}^{n-1}$), when $s=1/2$ the operator $L_s$ reduces to the Laplacian and in that case we are back to the standard thin obstacle problem.\\

All these results require that $M$ be a hyperplane and do not cover the case of an operator with variable coefficients, the aim of the present work is to extend the regularity theory to this situation in the case when $s=1/2$ (the classical thin obstacle problem). We introduce a monotonicity formula that is a bit different from the one in \cite{CSS08}, this formula will be used in future work to study the regularity of the free boundary using classical techniques (see for instance \cite{ACS06} or again \cite{CSS08}).

\section{General discussion}

In order to prove theorem \ref{thm:main} we only need to look at the case $f=0$, to see this, let $u$ be the solution of (\ref{thinobs}) for a non-zero $f$, and let $w$ solve

\begin{equation}
	\begin{array}{rl}
	 \Delta w = f & \mbox{ in } \Omega \\
	 w = 0 & \mbox{ on } \partial \Omega\\
	\end{array}\end{equation}

then $v=u-w$ solves a new thin obstacle problem where the right hand side is zero and where the thin obstacle is given by $\psi=\phi-w$ on the same hypersurface $M$.\\

As it is now standard, the regularity of $u$ will follow from estimates for the Neumann problem, namely, our goal is to show that if $u_\nu$ is the Neumann data for $u$ on $M$ (from either side of $M$) then $u_\nu$ is a $C^{1/2}$ function on $M$, more so, we only need to check this near free boundary points. Since this is an interior type result, we may take $\Omega = B_1$ and assume $0$ is a free boundary point, in fact, we are going to look always at the following normalized situation.\\

\begin{definition}
 
We say $u$ is a normalized solution to the thin obstacle problem if $u$ has the following properties: i) it solves the thin obstacle problem in $B_1$ with $\int_{B_1} u^2dx\leq 1$, ii) the hypersurface $M$ goes through the origin and the normal to $M$ there is $e_n$, iii) the origin is a free boundary point, and $\phi$ satisfies $\phi(0)=0$ and $\nabla_\tau \phi (0) = 0$ (here and from now on $\nabla_\tau$ will denote the gradient on $M$), iv) we also have $a_{ij}(0)=\delta_{ij}$ .\\

\end{definition}

Note that we can always assume we are in a renormalized situation: we can substract from $u$ and $\phi$ the linear part of $\phi$ at $0$ (which is a well-defined linear function in $B_1$ even though $\phi$ is only defined in $M$), this will change the boundary data and the obstacle, but it will not change the set of free boundary points. All we need to show now is that any normalized solution $u$ separates from $0$ like $|x|^{1+1/2}$, then we can conclude that a general solution $u$ separates from its linear part at a free boundary point like $|x|^{1+1/2}$, and therefore $u_\nu$ is of class $C^{1/2}$ on $M$.\\

In all that follows we will also use the following notation: assume $\mathbb{R}^n - M$ has two connected components, then we will call them $M^+$ and $M^-$. Observe that any $M$ that has some boundary may be extended to a bigger $M'$ without boundary, such that it separates $\mathbb{R}^n$ in two connected components and such that $M \cap B_2 = M' \cap B_2$. So when we want to concentrate on what happens on ``one of the sides of M'' we will just look at the intersection of $M^+$ with some other set. Additionally, $\nu$ will be used for the normal to any surface where we are performing integration (and the orientation will be clear from the context).

The paper is organized as follows: in section 3 we recall the main results from \cite{C79} and \cite{AC04} in order to show via compactness that solutions to our problem are $C^{1,\alpha}$ for any $\alpha<1/2$, in section 4 we present Almgren's monotonicity formula or its modification and prove an ``almost monotonicity'' lemma, in section 5 we use the monotonicity formula and the known regularity of global solutions to study blow ups of normalized solutions, finally in section 6 we show that normalized solutions decay as we want near free boundary points and this will complete the proof of theorem 1.

\section{Almost optimal regularity revisited}

Using the compactness method we can obtain the  ``almost''- $C^{1,1/2}$ regularity for our problem. The main idea is using the optimal regularity for blow-up solutions and pay a little price to get regularity for the non-blow up case, seeing it as a perturbation of the blow-up situation. The approach here follows the proof of interior estimates for elliptic equations using the regularity of harmonic functions.

\begin{lemma}\label{lem:compactness}
Let $u$ be a normalized solution to the thin obstacle problem in $B_1$ with obstacle given by $\phi$ on $M$. Then $\forall \epsilon>0$ there exists a $\delta>0$ such that if 

$$M \subset B_1 \cap \{ |x_n|\leq \delta\}$$ 

$$||a_{ij}-\delta_{ij}||_\infty \leq \delta$$

and $||\phi||_\infty \leq \delta$, $||a_{ij}||_{C^{1,\gamma}}\leq 1$ in $B_1\cap M$ then there exists a function $h$ which solves the thin obstacle problem for the Laplacian in $B_{1/2}$, and for this $h$ the thin obstacle is given by the null function on the hyperplane $x_n=0$ and we also have

$$\sup \limits_{B_{1/2}} |u-h| \leq \epsilon$$ 
	
$$h(0)=0, \nabla h(0)=0$$

Additionally, there exists a universal constant $C$ such that $||h||_{H^1(B_{1/2})} \leq C$.

\end{lemma}

\begin{proof}

Suppose that for some $\epsilon$ we cannot find a $\delta$ with the desired property, then for each $k$ we can find a normalized solution $u_k$ with obstacle $\phi_k$ defined on $M_k$ such that for any $h$ that is (global, flat thin obstacle) solution of the thin obstacle problem we have

$$\sup \limits_{B_r} |u_k-h| > \epsilon$$

and

$$||\phi||_\infty < 2^{-k}$$

$$ M \cap B_1 \subset \{ |x_n| < 2^{-k} \}$$

$$||a^{(k)}_{ij}-\delta_{ij}||_\infty < 2^{-k}$$

observe that the sequence $\{u_k\}$ is bounded in $H^1$, thus by a result of Caffarelli \cite{C79} we know that at least for some small $\alpha_0>0$ each $u_k$ is $C^{1,\alpha_0}$ on $B_{3/4} \cap M^{\pm}_k$ (see section 2 for notation) and these $C^{1,\alpha_0}$ norms are uniformly bounded in $k$. Therefore we can pick a subsequence (wich we also call $u_k$) converging uniformly in $B_{1/2}$ to some function $h \in H^1$, and such that $\nabla u_k \to \nabla h$ uniformly in $B_{1/2}$. Now the uniform $C^{1,\alpha_0}$ regularity of the $u_k$ and the fact that $\nabla u_k(0) = 0$ for each $k$ (remember $0$ was a free boundary point for each of them) shows that $\nabla h(0)=0$.\\

Now the assumptions on $M_k,\phi_k$ and $a^{(k)}_{ij}$ force the limit $h$ to solve the thin obstacle problem for the Laplacian with $M_0 = \{ x_n = 0 \}$, $\phi_0 \equiv 0$. Since $u_k \to h$ uniformly in $B_{1/2}$ we have for $k$ large enough that

$$\sup \limits_{B_1/2}|u_k-h| \leq \epsilon$$ 

since $h$ has all the desired properties we have a contradiction, and this proves the lemma.

\end{proof}

\begin{lemma}\label{lem:keyiter}
Given any $\alpha$ with $0<\alpha<1/2$ there exist constants $\lambda $ $(0<\lambda<1)$ and  $\delta_0>0$  such that given any normalized solution $u$ of the thin obstacle problem for which $||\phi||_\infty \leq \delta_0$  and $M \subset B_1 \cap \{|x_n|\leq \delta_0\}$ we have

$$\sup_{B_\lambda} |u| \leq \lambda^{1+\alpha}$$

\end{lemma}

\begin{proof}

For a constant $\epsilon>0$ to be specified later, let $\delta$ and $h$ be as in the previous lemma, then 

$$\sup \limits_{B_1/2}|u-h| \leq \epsilon$$ 

then by the result of Athanasopoulos and Caffarelli \cite{AC04} we know that $h$ is of class $C^{1,1/2}$ on either side of $\{x_n=0\}$, moreover, since $h$ also satisfies $\nabla h(0) = 0$, we have for $r$ small 

$$\sup \limits_{B_r}|h| \leq Cr^{1+1/2}$$

here $C$ is a universal constant (determined by the universal constant in the previous lemma), then we have for every $\lambda <1/2$ that

$$\sup \limits_{B_\lambda} |u| \leq \sup \limits_{B_\lambda} |u-h|+ \sup \limits_{B_\lambda} |h|$$

$$ \leq \epsilon + C\lambda^{1+1/2} $$

by the assumption $\alpha$ is a positive number with $\alpha<1/2$, then we can pick $\lambda$ so that 

$$C\lambda^{1+1/2} \leq \frac{1}{2}\lambda^{1+\alpha}$$

i.e. by taking $\lambda$ small enough so that $C\lambda^{\frac{1}{2}-\alpha} < \frac{1}{2}$, then if we pick $\epsilon = \frac{1}{2}\lambda^{1+\alpha}$ we obtain

$$\sup \limits_{B_\lambda} |u| \leq \lambda^{1+\alpha}$$

then the lemma is proved by letting $\delta_0=\delta$.

\end{proof}

\begin{theorem}\label{thm:almostopt}

Under the same assumptions as theorem 1.1 the solution $u$ is of class $C^{1,\alpha}$ on either side of the hypersurface $M$ and away from $\partial B_1$ for every $\alpha$ such that $\alpha<1/2$. In particular, since $\beta >1/2$, $u$ must be of class $C^{1,\alpha}$ for some $\alpha$ for which $\alpha+\beta>1$.

\end{theorem}

\begin{proof}

Fix $\alpha$ with $0<\alpha<1/2$, let $u$ be a normalized solution, by Neumann estimates all we need to show is that $u$ decays like $|x|^{1+\alpha}$. By a rescaling of the form $u(x) \to u(tx)$ for some small $t$ we may assume without loss of generality that $\phi$ and $M$ satisfy the conditions of the previous lemma, therefore for some $\lambda<1$

$$\sup \limits_{B_\lambda}|u| \leq \lambda^{1+\alpha}$$
 
define then $u_1(x)=\frac{u(\lambda x)}{\lambda^{1+\alpha}}$, it is straightforward to check that $u_1$ is a normalized solution in $B_1$, the thin obstacle is given by some $M_1$ and $\phi_1$ which again satisfy the conditions of the previous lemma, and this is thanks to the fact that both are of class $C^{1,\beta}$; a similar argument applies to the coefficients $a^{(1)}_{ij}$, reapplying the previous lemma we see that

$$\sup \limits_{B_\lambda} |u_1| \leq \lambda^{1+\alpha}$$ 

which is the same as

$$\sup \limits_{B_{\lambda^2}}|u|\leq \lambda^{2(1+\alpha)}$$

by iterating this argument $k$ times we conclude

$$\sup \limits_{B_{\lambda^k}} |u| \leq \lambda^{2(1+\alpha)}$$

finally, for each every $r$ with $0<r\leq 1$ there is a $k$ such that $\lambda^{k+1} \leq r \leq \lambda^k$, thus

$$\sup \limits_{B_r}|u| \leq \sup \limits_{B_{\lambda^k}}|u| \leq \lambda^{-(1+\alpha)}\lambda^{(1+k)(1+\alpha)}$$

i.e. $\sup \limits_{B_r}|u| \leq Cr^{1+\alpha}$ for some universal $C$, this finishes the proof.

\end{proof}

\section{Monotonicity formula}

Given a normalized solution $u$ to the \emph{Signorini problem } we will study the quantity 

$$F_u(r):=\int_{S_r}u^2d\sigma$$

our main objective is showing that $F_u(r)$ decays at an optimal rate, this averaged rate of decay will imply a strong decay thanks to the fact that $u$ solves an elliptic equation (this will be done in detail in section 6). Our main tool in the study $F_u$ is a version of Almgren's monotonicity formula.\\

Recall that by the almost optimal regularity, we may assume $u$ is a $C^{1,\alpha}$ function on $B_{1/2} \cap M^\pm$ for some $\alpha$ such that $\alpha+\beta=1+\epsilon_0$ for some small $\epsilon_0>0$, this fact will be used extensively in this section.

\begin{theorem}{(Monotonicity formula)}\label{thm:monotonicity}

For a normalized solution $u$, the quantity

$$\Phi_u(r) = r\frac{d}{dr}\log \max (F_u(r),r^{n+2})$$

is almost monotone, specifically, there is a universal constant $C>0$ such that

$$\Phi'_u(r)\geq-Cr^{-1+\epsilon_0}\Phi(r)$$

\end{theorem}

The proof involes a technical estimate which we state as a separate lemma.

\begin{lemma}\label{lem:keylem} Under the same assumptions as before, there is a universal $C>0$ such that for each $r$ for which $F_u(r) > r^{n+2}$ we have the estimate

$$\frac{\int_{S_r}|\nabla u|^2d\sigma}{\int_{B_r}|\nabla u|^2dx} \geq \frac{n-2}{r}+ \frac{\int_{S_r}2u_r^2d\sigma}{\int_{S_r}uu_rd\sigma}-Cr^{-1+\epsilon_0}$$ 
 
\end{lemma}

Remark: As shown in Lemma 1 of \cite{ACS06}, when $M$ is a hyperplane and $\phi$ is identically zero we have

$$\frac{\int_{S_r}|\nabla u|^2d\sigma}{\int_{B_r}|\nabla u|^2dx} =\frac{n-2}{r}+ \frac{\int_{S_r}2u_r^2d\sigma}{\int_{S_r}uu_rd\sigma}$$ 

which they use to prove that $\Phi_u(r)$ is monotone, the lack of simmetry (i.e. translation/scale invariance) accounts for the appearance of the ``error'' term $-r^{-1+\epsilon_0}$ in the present case. 

\begin{proof}{(of Lemma ~\ref{lem:keylem})}\\

Consider the identity

$$\mbox{div}(|\nabla u|^2x-2(\nabla u \cdot x)\nabla u) = (n-2)|\nabla u|^2-2(\nabla u \cdot x) \Delta u$$

by the $C^{1,\alpha}$ regularity of $u$ and the regularity of the coefficients $a_{ij}$ it can be seen that $\Delta u = O(r^\alpha)$, so that the last term above is $O(r^{\alpha+\beta+1})$, then by Stokes theorem

$$\int_{S_r \cap M^+}(|\nabla u|^2x-2(\nabla u \cdot x)\nabla u) \cdot \nu d\sigma=\int_{B_r \cap M^+}(n-2)|\nabla u|^2dx+O(r^{n+1+\alpha+\beta})$$

the same applies to $B_r \cap M^-$, adding the two formulas we get a term that consists of integrating over the hypersurface $M$, then taking into account the continuity of tangential derivatives across $M$ we get

$$r\int_{S_r}|\nabla u|^2-2u_r^2d\sigma-\int_{M \cap B_r}[u_\nu^2](x \cdot \nu)+2(\nabla_\tau u \cdot x)[u_\nu]d\sigma=(n-2)\int_{B_r}|\nabla u|^2+O(r^{n+1+\alpha+\beta})$$

which we rewrite as

$$\int_{S_r}|\nabla u|^2d\sigma=\int_{S_r}2u_r^2d\sigma+\frac{1}{r}\int_{M \cap B_r}[u_\nu^2](x \cdot \nu)+2(\nabla_\tau u \cdot x)[u_\nu]d\sigma$$

\begin{equation}\label{eqn:keylem1} +\frac{(n-2)}{r}\int_{B_r}|\nabla u|^2+O(r^{n+\alpha+\beta})
\end{equation}

To estimate the integral on $M$, observe first that $(\nabla_\tau u \cdot x)[u_\nu] = O(r^{\alpha+\beta+1})$. Indeed, when $[u_\nu] \neq 0$ we have $u = \phi$, so there $u$ is $C^{1,\beta}$ and since $\nabla u(0) = 0$ there we have  $|\nabla_\tau u(x)|\leq C|x|^\beta$ for some $C$. On the other hand, $|[u_\nu]|\leq Cr^\alpha$ by the $C^{1,\alpha}$ regularity of $u$, thus

$$|(\nabla_\tau u \cdot x)[u_\nu]| \leq |\nabla_\tau u| |x| |[u_\nu]| \leq C|x|^{\alpha+\beta+1}$$

Next, note that since $M$ is $C^{1,\beta}$ $(x\cdot \nu) = O(r^{1+\beta})$, and once again by the $C^{1,\alpha}$ regularity of $\alpha$ we know that $[u_\nu^2]=O(r^{\alpha})$, thus we know that the integral over $M$ is not too big, that is

$$\int_{M \cap B_r}[u_\nu^2](x \cdot \nu)+2(\nabla_\tau u \cdot x)[u_\nu]d\sigma=O(r^{n+\alpha+\beta})$$

plugging this in (\ref{eqn:keylem1})

$$\int_{S_r}|\nabla u|^2d\sigma=\int_{S_r}2u_r^2d\sigma+\frac{(n-2)}{r}\int_{B_r}|\nabla u|^2+O(r^{n+\alpha+\beta-1})$$

which is the same as

$$\frac{\int_{S_r}|\nabla u|^2d\sigma}{\int_{B_r}|\nabla u|^2dx}=\frac{n-2}{r}+\frac{\int_{S_r}2u_r	^2d\sigma}{\int_{B_r}|\nabla u|^2dx}+\frac{O(r^{n+\alpha+\beta-1})}{\int_{B_r}|\nabla u|^2dx}$$

Now recall that for this particular $r$ we have $F_u(r)>r^{n+2}$, therefore by Poincar\'e's inequality we get  $\int_{B_r}|\nabla u|^2dx\geq Cr^{n+1}$ for some universal $C$. Then we have

\begin{equation}\label{eqn:keylem2}\frac{\int_{S_r}|\nabla u|^2d\sigma}{\int_{B_r}|\nabla u|^2dx}\geq\frac{n-2}{r}+\frac{\int_{S_r}2u_r	^2d\sigma}{\int_{B_r}|\nabla u|^2dx}-Cr^{-1+\epsilon_0}\end{equation}

Next, we deal with the term $\int_{B_r}|\nabla u|^2dx$. By Stokes' theorem we have that

$$\int_{B_r \cap M^+}|\nabla u|^2dx=\int_{S_r \cap M^+}uu_rd\sigma+\int_{B_r \cap M}u u_{\nu_+}d\sigma-\int_{B_r \cap M^+}u\Delta udx$$

adding the corresponding identity for $B_r \cap M^-$ we obtain

$$\int_{B_r}|\nabla u|^2dx=\int_{S_r}uu_rd\sigma + \int_{B_r \cap  M}2u [u_{\nu}]d\sigma -\int_{B_r}u\Delta u dx $$

now we estimate the last two integrals on the right. Since for $u>\phi$ we have $[u_\nu]=0$ the boundary integral is only non-zero where $u=\phi$, i.e. where $u$ is $C^{1,\beta}$, and as before $[u_\nu]=O(r^\alpha)$, then

$$\int_{B_r \cap M}2u [u_{\nu}]d\sigma = O(r^{n+\alpha+\beta})$$

furthermore, $u$ decays like $r^{1+\alpha}$ in general, and we already know that $\Delta u = O(r^\alpha)$ so

$$\int_{B_r}u\Delta u dx = O(r^{n+1+2\alpha})$$

we put all this together now and get

$$\int_{B_r}|\nabla u|^2dx=\int_{S_r}uu_rd\sigma +O(r^{n+\alpha+\beta}) $$

In particular, since we already know the integral on the left is bounded from below by $Cr^{n+1}$, we conclude that

\begin{equation}\label{eqn:keylem3}\int_{S_r}uu_rd\sigma \geq Cr^{n+1}\end{equation}

We arrive to the equation

$$\frac{\int_{S_r}2u_r^2d\sigma}{\int_{B_r}|\nabla u|^2dx}=\frac{\int_{S_r}2u_r^2d\sigma}{\int_{S_r}uu_rd\sigma + O(r^{n+\alpha+\beta})}$$

which using the Taylor expansion for $(1-x)^{-1}$ and the inequality (\ref{eqn:keylem3}) leads to

$$\frac{\int_{S_r}2u_r^2d\sigma}{\int_{B_r}|\nabla u|^2dx}=\frac{\int_{S_r}2u_r^2d\sigma}{\int_{S_r}uu_rd\sigma}+O(1)$$

putting this together with (\ref{eqn:keylem2}) we end up with

$$\frac{\int_{S_r}|\nabla u|^2d\sigma}{\int_{B_r}|\nabla u|^2dx}=\frac{n-2}{r}+\frac{\int_{S_r}2u_r^2d\sigma}{\int_{S_r}uu_rd\sigma}+O(r^{-1+\epsilon_0})$$

and this finishes the proof. 

\end{proof}

\begin{proof}{(of Theorem ~\ref{thm:monotonicity})}\\

Consider the function $\Phi_u(r)=r \frac{d}{dr}  \log \max\{F_u(r),r^{n+2}\}$, to estimate $\Phi'_u(r)$ from below we can concentrate in those $r$'s for which $F_u(r) > r^{n+2}$, then in a neighborhood of any such $r$ we have 

$$\Phi_u(r)=r\frac{\int_{S_r}2uu_nd\sigma}{\int_{S_r}u^2d\sigma}+n-1$$

applying Stokes theorem on both sides of $M$ this gives

\begin{equation}\label{eqn:errorterm}
 \Phi_u(r)=r\frac{\int_{B_r}|\nabla u|^2dx}{\int_{S_r}u^2d\sigma}+\frac{\int_{M \cap B_r}2u[u_n]d\sigma}{\int_{S_r}u^2d\sigma}+n-1
\end{equation}

therefore

$$\frac{d}{dr}\log(\Phi_u(r))=\frac{1}{r}+\frac{\int_{S_r}|\nabla u|^2d\sigma}{\int_{B_r}|\nabla u|^2dx}-\frac{(n-1)r^{-1}\int_{S_r}u^2d\sigma+\int_{S_r}2uu_rd\sigma}{\int_{S_r}u^2d\sigma}$$

$$+\frac{d}{dr}\left ( \frac{\int_{M \cap B_r}2u[u_n]d\sigma}{\int_{S_r}u^2d\sigma} \right )$$
	
$$\frac{d}{dr}\log(\Phi_u(r))\geq\frac{2-n}{r}+\frac{\int_{S_r}|\nabla u|^2d\sigma}{\int_{B_r}|\nabla u|^2dx}-\frac{\int_{S_r}2uu_rd\sigma}{\int_{S_r}u^2d\sigma}-Cr^{-1+\epsilon_0}$$

moreover, lemma \ref{lem:keylem} tells us that

$$\frac{\int_{S_r}|\nabla u|^2d\sigma}{\int_{B_r}|\nabla u|^2dx} \geq \frac{n-2}{r}+ \frac{\int_{S_r}2u_r^2d\sigma}{\int_{S_r}2uu_rd\sigma}-Cr^{-1+\epsilon_0} $$

therefore

$$\frac{d}{dr}\left ( \log (\Phi_u(r)\right ) \geq \frac{\int_{S_r}2u_r^2d\sigma}{\int_{S_r}2uu_rd\sigma}-\frac{\int_{S_r}2uu_rd\sigma}{\int_{S_r}u^2d\sigma}-Cr^{-1+\epsilon_0}$$

by Cauchy-Schwartz, the difference on the right is always non-negative, and we conclude that there exist constants $C>0$ and $r_0>0$ such that for all $r<r_0$ we have

$$\frac{d}{dr}\left ( \log (\Phi_u(r)\right )\geq \frac{\int_{S_r}2u_r^2d\sigma}{\int_{S_r}2uu_rd\sigma}-\frac{\int_{S_r}2uu_rd\sigma}{\int_{S_r}u^2d\sigma} -Cr^{-1+\epsilon_0} \geq -Cr^{-1+\epsilon_0}$$

tracing the constants and the ``$O(.)$''estimates the dependence of $C$ can be checked easily.
 
\end{proof}

\begin{corollary}\label{cor:muest}
There exists a constant $C>0$ such that for all $r<1$ we have

$$\Phi_u(r) \geq \mu-Cr^\epsilon_0$$

where $\mu = \liminf \limits_{s\to 0^+} \Phi_u(s)$

\end{corollary}

The proof of \ref{thm:monotonicity} shows almost at once the following slightly more general result

\begin{corollary}\label{cor:monotonicity2}
Under the same assumptions as theorem \ref{thm:monotonicity} consider the function

$$\Phi^{(\delta_0)}_u(r) = r\frac{d}{dr}\log \max (F_u(r),r^{n+2+\delta_0})$$

Then, if $\delta_0<\beta-\frac{1}{2}$ the almost monotonicity estimate from theorem \ref{thm:monotonicity} holds for $\Phi^{(\delta_0)}_u$ for some (smaller) $\epsilon_0$.

\end{corollary}

Indeed, one only needs to note that everytime we divided by terms controlled from below by $r^{n+2}$ there was room for a factor of $r^{-\delta_0}$, for a small $\delta_0$. Thus the final estimate remains as long as we pick $\delta_0+\epsilon_0$ smaller than $\alpha+\beta-1$.

\section{Blowup estimates}

We state the main result of this section

\begin{theorem}\label{thm:blowup}

Let $u$ be a normalized solution to the Signorini problem. Then 

$$\liminf \limits_{s \to 0^+} \Phi_u(s) \geq n+2$$

\end{theorem}
 
We will prove this by a blowup argument, so assume $u$ is a normalized solution and let 

$$d_r = \left ( r^{-(n-1)}\int_{S_r}u^2d\sigma \right )^{1/2} = \left ( r^{-(n-1)}F_u(r)\right )^{1/2}$$

and define $u_r(x) := \frac{1}{d_r}u(rx)$ for $0 < r \leq 1$. Observe that $u_r$ is still normalized solution, and note by changing variables that $\Phi_u(r)=\Phi_{u_r}(1)$. The following argument is a variation of the one from Lemma 6.2 in \cite{CSS08}.

\begin{lemma}\label{lem:blowup}
Suppose $r^{-3/2}d_r \to +\infty$ as $r \to 0^+$, then there exists a sequence $r_k>0$ with $r_k\to 0$ as $k \to \infty$, and a nonzero $u_0\in H^1(B_1)$ such that (writing $u_k$ for $u_{r_k}$)

$$u_k \to u_0 \mbox{ in } H^1(B_{1/2})$$

Where $u_0$ solves the thin obstacle problem for a null obstacle on a hyperplane and it is a homogeneous function of degree $(\Phi_u(0)-n+1)/2$.

\end{lemma}

\begin{proof}

First we shall prove $u_r$ is bounded in $H^1$, by the assumption, there exists $r_0>0$ such that $F_u(r)>r^{n+2}$ for $r<r_0$, i.e. $\Phi_u(r) = r\frac{d}{dr}\log F_u(r)$. Since $\Phi_u(r)$ is bounded for small $r$, there exists a number $C>0$ such that 

$$\Phi_u(r)=r\frac{\int_{B_r}2uu_ndx}{\int_{S_r}u^2d\sigma}+n-1 \leq C$$

whenever $r<r_0$. If we know apply Stoke's theorem on both sides of $M$ we are left with

$$\int_{S_r}2uu_nd\sigma = \int_{B_r}|\nabla u|^2dx+\int_M 2u[u_n]d\sigma$$

then, thanks to the fact that $F_u(r)>r^{n+2}$ and that $u$ is $C^{1,\alpha}$ for some $\alpha$ such that $\alpha+\beta>1$, we get that

$$r\frac{\int_{B_r}2uu_ndx}{\int_{S_r}u^2d\sigma} \geq r\frac{\int_{B_r}|\nabla u|^2dx}{\int_{S_r}u^2d\sigma}-Cr^{\epsilon}$$

where $C$ depends on the $C^{1,\alpha}$ norm of $u$ and the $C^{1,\beta}$ norm of $M \cap B_1$. Now doing the	 change of variables $y=rx$ and using the definition of $u_r$ we see that for some $C$ and for all $r<r_0$ we have

$$\int_{B_1}|\nabla u_r|^2dx \leq C \mbox{ and } \int_{S_1}u_r^2d\sigma =1$$

therefore the sequence $\{u_r\}$ is bounded in $H^1(B_1)$, it can be checked (as done in \cite{CS06}) that there is a subsequence (renamed $u_k$) that converges to some non-zero $u_0$ in $H^1$.\\

We now prove that for all $s$ we have

\begin{equation}\label{eqn:blowup1}
\Phi_{u_0}(s)=\Phi_u(0)
\end{equation}

for this we apply the rescaling $x\to r_kx$ to the formula for $\Phi_{u_k}$, we obtain for all $s>0$ and every $k$

\begin{equation}
 \Phi_{u_k}(s)=s\frac{\int_{B_s}2u_k(u_k)_rdx}{\int_{S_s}u_k^2d\sigma}+n-1=\Phi_u(r_ks)
\end{equation}

moreover we can show that for fixed $s$

$$\lim \limits_{k \to \infty}s\frac{\int_{B_s}|\nabla u_k|^2dx}{\int_{S_s}u_k^2d\sigma}=s\frac{\int_{B_s}|\nabla u_0|^2dx}{\int_{S_s}u_0^2d\sigma}$$

and this proves (\ref{eqn:blowup1}) if we note that the ``error term'' (see equation (\ref{eqn:errorterm}))

$$r\frac{\int_{M \cap B_r}2u[u_n]d\sigma}{\int_{S_r}u^2d\sigma}$$

goes to zero as $r \to 0$ in the present case.\\

What we are left with now is a sequence $u_k$ in $H^1$ satisfying

\begin{equation}
	\begin{array}{rl}
	u_k \geq \phi_k & \mbox{on } M_k \cap B_1 \\
	\Delta u_k \leq 0 & \mbox{ in } B_1 \\
	\Delta u_k = 0 & \mbox{away from } \{x \in M_k|u_k(x)=\phi_k(x)\}\\
	\end{array}\end{equation}

we conclude that $u_0$ solves 

\begin{equation}
	\begin{array}{rl}
	u_0 \geq 0 & \mbox{on } M_0 \cap B_1 \\
	\Delta u_0 \leq 0 & \mbox{ in } B_1 \\
	\Delta u_0 = 0 & \mbox{away from } \{x \in M_0|u_0(x)=\phi_0(x)\}\\
	\end{array}\end{equation}

here $M_0$ is the tangent hyperplane to $M$ at 0. Now Almgren's monotonicity formula in its original form may be applied directly to $u_0$: since (7) says that the monotonicity formula for $u_0$ is constant then it must be that $u_0$ is homogeneous of degree $(\Phi_u(0)-n+1)/2$, and this concludes the proof.

\end{proof}

With the lemma proven we proceed to finish the proof of Theorem \ref{thm:blowup}.

\begin{proof}

Assume first that

$$\limsup \limits_{r \to 0}\frac{d_r}{r^{3/2}} =  +\infty $$

 then lemma \ref{lem:blowup} tells us that there exists a homogenous function $u_0 \in H^1$ with degree $(\Phi(0)-n+1)/2$ that solves the thin obstacle problem for a null osbtacle on a hyperplane, in this case  Athanasopoulos and Caffareli proved that $u_0$ is of class $C^{1,1/2}$ on either side of the hyperplane, therefore the degree of $u_0$ must be no less than $3/2$, so we must have

$$\Phi_u(0)\geq n+2$$

The second case is dealt with a standard argument, which we include for completeness. Since $\limsup r^{-3/2}d_r < + \infty$ we have $d_r \leq Cr^{3/2}$, for some $C$; if we had $d_{r_k} \leq {r_k}^{3/2}$ for a sequence $r_k$ going to zero, it would follow that at once that $\Phi_u(0)=n+2$. Then we may assume we have $d_r \geq r^{3/2}$ for all small $r$, that is, we have for all small $r$ (using the definition of $d_r$)

\begin{equation}\label{eqn:case2}
r^{n+2}\leq F_u(r) \leq C^{n+2}
 \end{equation}

suppose now that there exists a sequence $\{r_k\}$ ($r_k >0$) that is going to zero and a small $\epsilon_0>0$ for which we have $\Phi_u(r_k)\leq n+2-\epsilon_0$ for all $k$. Taking the logarithm in \ref{eqn:case2} we have for all $k,l$ with $l>k$

$$(n+2)(\log r_l - \log r_k)-C \leq \log F_u(r_k) - \log F_u(r_l) $$

$$= \int_{r_l}^{r_k}\frac{d}{dr}\log F_u(r) dr \leq (n+2-\epsilon_0)(\log r_k - \log r_l)$$

the last inequality shows that $n+2 \leq n+2 -\epsilon_0$ by just taking $l$ big enough so that $\log r_k - \log r_l > 0$, so we have a contradiction and the theorem is proved.

\end{proof}

\section{Proof of Theorem 1}

Finally we prove $u$ decays as desired near free boundary points, the rest of the proof is word by word just as in \cite{CSS08}, we present it here for completeness.

\begin{lemma}\label{lem:Fdecay}
Suppose $\liminf \limits_{s \to 0^+}\Phi_u(s)\geq \mu$, then $F_u(r)\leq Cr^\mu$.
\end{lemma}

\begin{proof}
 
Let $\hat{F}(r)=\max(F_u(r),r^{n+2})$, by corollary \ref{cor:muest} there are positive constants $C$ and $\epsilon$ such that

$$\frac{d}{dr}\log \hat{F}(r) \geq \frac{\mu}{r}-Cr^{-1+\epsilon}$$

integrating this inequality from $r$ to $1$ we obtain

$$\log \hat{F}(1) - \log \hat{F}(r) \geq -\mu \log r -C(r^\epsilon+1)$$

which is the same as

$$\log \hat{F}(r) \leq \mu \log r + O(1)$$

taking the exponential on both sides we get for some constant $C$

$$\hat{F}(r) \leq Cr^\mu$$

i.e. $F_u(r)\leq Cr^\mu$, and the assertion is proved.

\end{proof}

The following lemma says that a normalized solution decays the way we want near every free boundary point, and 	by the discussion in section 2 this will prove Theorem \ref{thm:main}.

\begin{lemma}\label{lem:avrgdecay}
Let $u$ be a normalized solution of the thin obstacle problem, suppose that for some $C$ we have $F_u(r) \leq Cr^{n+2}$ for all small $r$, then there is a constant $C_1$ such that for all small $r$

$$\sup \limits_{B_r(0)} |u| \leq C_1r^{1+1/2}$$

\end{lemma}	

\begin{proof}

Let $x_0 \in B_r(0)$, we claim that  that there is a constant $C$ such that

\begin{equation}
\left | \frac{1}{\omega_n \tau^{n}}\int_{B_\tau(x_0)}u^2d\sigma-u(x_0)^2 \right | \leq C\tau^{3}
\end{equation}

this estimate will follow from the same estimate where the average on $B_\tau(x_0)$ is replaced by the average on $\partial B_\tau(x_0)$, then we have (as it is usually done for instance when proving the mean value theorem)

\begin{equation}\frac{1}{\alpha_n\tau^{n-1}}\int_{\partial B_\tau(x_0)}u^2d\sigma-u(x_0)^2=\int_0^\tau\frac{1}{\alpha_nt^{n-1}}\int_{\partial B_t(x_0)}2uu_nd\sigma dt\end{equation}

but now, using integration by parts we can see that

$$\int_{\partial B_\tau(x_0)}2uu_nd\sigma=\int_{B_\tau(x_0) \cap M}2u[u_\nu]d\sigma$$

then arguing as in the proof of Lemma \ref{lem:keylem} we see that

$$\left |\int_{B_\tau(x_0) \cap M}2u[u_\nu]d\sigma \right |\leq Cr^{n-1+(1+\alpha+\beta)}$$

since we have an $\alpha$ for which $\alpha+\beta>1$ we conclude that the right hand side in (13) is controlled in absolute value by $C\tau^3$, that is, we have

$$\left | \frac{1}{\alpha_n\tau^{n-1}}\int_{\partial B_\tau(x_0)}u^2d\sigma-u(x_0)^2 \right | \leq C\tau^3$$

and this proves (12), on the other hand, if we take $|x_0|\leq r$, then $B_{r}(x_0) \subset B_{2r}(0)$ so	

\begin{equation}\frac{1}{\omega_n r^{n}}\int_{B_r(x_0)}u^2d\sigma \leq \frac{2^n}{\omega_n (2r)^{n}}\int_{B_{2r}(0)}u^2d\sigma\end{equation}

but we know the decay of the right hand side thanks to our assumption on $F_u(r)$, so it is less than $Cr^3$ (as can be seen integrating in spherical coordinates). Putting this together, we can estimate the size of $u(x_0)^2$

$$u(x_0)^2\leq \left | \frac{1}{\omega_nr^n}\int_{B_r(x_0)}u^2d\sigma-u(x_0)^2 \right |$$

$$+\frac{1}{\omega_n r^{n}}\int_{B_r(x_0)}u^2d\sigma$$

and by (12)

$$u(x_0)^2 \leq Cr^3 + \frac{2^n}{\omega_n (2r)^{n}}\int_{B_{2r}(0)}u^2d\sigma \leq C_1r^3$$

and this proves the lemma.

\end{proof}

Now we can prove the optimal regularity: Theorem 4 allows us to apply Lemma 5 with $\mu=n+2$, we just plug this in Lemma 6 and we get the desired decay of normalized solutions at free boundary points, and this completes the proof of Theorem 1.

\begin{acknowledgements}
This work was done as part of my doctoral studies under the supervision of Prof. Luis Caffarelli, I would like to thank him for his encouragement and guidance. I would also like to express my sincere thanks to Russell Schwab and Luis Silvestre for proofreading an earlier draft of this manuscript.

\end{acknowledgements}

\bibliographystyle{amsplain}
\bibliography{Signorini_Almgren}   

\end{document}